\documentclass{article}
 \usepackage{amsmath,latexsym, amsthm, amsfonts, amssymb, amsxtra,amscd, enumerate,color, cancel}

\usepackage{bbm, dsfont,latexsym} 
\usepackage{float,wrapfig}
\usepackage[english]{babel}
\usepackage[usenames,dvipsnames,svgnames,table]{xcolor}
\usepackage{times, graphicx,stmaryrd}  
\RequirePackage[colorlinks,citecolor=blue,urlcolor=blue]{hyperref}
\usepackage[active]{srcltx}

\newtheorem{theorem}{Theorem}

\newtheorem{proposition}{Proposition}
\newtheorem{remark}{Remark}
\newtheorem{definition}{Definition}
\newcommand{\F}{\mathcal{F}}
\newcommand{\FF}{\{\F_t\colon t\geq 0\}}%{\mathbb{F}}
\newcommand{\R}{\bf R}
\newcommand{\Po}{\mathbb{P}}
\newcommand{\E}{\mathbb{E}}

\newcommand{\red}{\color{black}}
\newcommand{\blue}{\color{black}}

\begin{document}
\title{Optimal stopping of Brownian motion with broken drift}%Shape-transitions in optimal stopping sets\\ of diffusions}
%\date{\today}
%\maketitle
\author{Ernesto  Mordecki\thanks{Centro de Matematica, Facultad de Ciencias, Igua 4225, 11400, Montevideo, Uruguay, e-mail: mordecki@cmat.edu.uy}  
\ and\   Paavo Salminen\thanks{\AA bo Akademi University, Faculty of Science
  and Engineering, FIN-20500 \AA bo, Finland, email:\break phsalmin@abo.fi.}}
\maketitle
\abstract{We solve an optimal stopping problem where the underlying diffusion is Brownian motion on $\R$ with a positive drift changing at zero. It is assumed that the drift $\mu_1$ on the negative side is smaller than the drift $\mu_2$ on the positive side.   The main observation is that if $\mu_2-\mu_1>1/2$ then there exists values of the discounting parameter for which it is not optimal to stop in the vicinity of zero where the drift changes. However, when the discounting gets bigger the stopping region becomes connected and contains zero. This is in contrast with results concerning optimal stopping of skew Brownian motion where the skew point is for all values of the discounting parameter in the continuation region.}\\



%\tableofcontents

\thispagestyle{empty} \clearpage \setcounter{page}{1}

\section{Introduction}

{\bf 1.1 Motivation.}
%%%%%%%%%%%%%%%%%%%%%%%%
%%%%%% motivation
%%%%%%%%%%%%%%%%%%%%%%%%
Optimal stopping problems for one-dimensional diffusions is a much studied topic and there exists a variety of methods for finding solutions. During recent decade some interest has been focused on cases where the underlying diffusion has exceptional  points such as: sticky points; skew points; discontinuities in the infinitesimal variance and/or drift.
%\begin{description}
%\item{$\bullet$} sticky points,
%\item {$\bullet$} skew points,
%\item {$\bullet$}  discontinuities in the infinitesimal variance and/or drift .
%\end{description} 
The main observation in the presence of sticky points is that  the smooth fit does not in general hold 
(see for instance Crocce \cite{Crocce},  Crocce and Mordecki  \cite{CrocceMordecki}, and Salminen and Ta \cite{SalminenTa}). 
If the diffusion has skew points then it is possible that in the  vicinity of such a point it is not optimal to stop for any value of the discounting parameter,  
as was found by \'Alvarez and Salminen \cite{AlvarezSalminen} and also by Presman \cite{Presman}. Explicit solutions of optimal stopping problems when the underlying diffusion has discontinuous coefficients have not, to our best knowledge, been encountered in the literature. Many of the methods to solve optimal stopping problems do not, however,  exclude such diffusions. This is, in particular, the case in the approach based on excessive functions, see, e.g., Salminen \cite{Salminen}, Dayanik and Karatzas \cite{DaKa}, Christensen and Irle \cite{ChIr}, and Crocce and Mordecki \cite{CrocceMordecki}.
Also the approach via variational inequalities and free boundary problems, see Lamberton and Zervos \cite{LamZer13} and Ruschendorf and Urusov \cite{RuschendorfUrusov},  does not seem to require continuity of the diffusion coefficients. However, our main motivation for the present work is to study how the stopping set changes as a function of the discounting parameter when, in our case, the drift is discontinuous.

\smallskip
\noindent
{\bf 1.2 Diffusions and the optimal stopping problem.}
Consider a non-terminating and regular one-dimensional (or linear) diffusion $X=\{X_t\colon t\geq 0\}$ 
in the sense of It\^o and McKean \cite{ItoMcKean} (see also Borodin and Salminen \cite{BorodinSalminen}). 
The state space of $X$ is denoted by ${\bf I}$, 
an interval of the real line $\R$  with left endpoint ${a}=\inf{\bf I}$ and right endpoint $b=\sup{\bf I}$,
where $-\infty\leq{a}<b\leq\infty$. The notations $m$ and $S$ are used for the speed measure and the scale function, respectively, of $X.$ Moreover, for $r\geq 0$ let $\varphi_r$ and $\psi_r$ denote the decreasing and increasing, respectively, fundamental solutions of the generalized ODE (see \cite{BorodinSalminen} II.10 p.18) 
\begin{equation}
\label{gen}
\frac d{dm}\frac d{dS} u=ru.
\end{equation}  
%We exclude the possibility of absorbing and killing boundaries; if some of the boundaries belong to ${\bf I}$ we assume it to be both entrance and exit (i.e. non-singular).
%We assume $r\not \in {\bf I}$.
Let $\Po_x$ stand for  the probability measure associated with $X$ when starting from $x$, 
and by $\E_x$ the corresponding mathematical expectation. 
Denote by $\mathcal{M}$ the set of all stopping times with respect to $\FF$, the usual augmentation of the natural 
filtration generated by $X.$ %(see i.14 in \cite{BorodinSalminen}).
%\subsection{OSP}
Given a non-negative lower semicontinuous reward function $g\colon{\bf I}\to \R$ and a discount factor ${r} \geq 0$, 
consider the optimal stopping problem consisting of finding a function $V_r$ and a stopping time $\tau^*\in\mathcal{M}$, such that
\begin{equation}\label{eq:osp}
V_r(x)=\E_{x}[{e^{-r\tau^*} g(X_{\tau^*})}] = \sup_{\tau \in \mathcal{M}}\E_{x}[{e^{-r\tau} g(X_{\tau})}].
\end{equation}
The \emph{value function} $V_r$ and the \emph{optimal stopping time} $\tau^*$  constitute  the solution of the problem.

\par
\smallskip
\noindent
{\bf 1.2  References.}
After the classical works of 
McKean \cite{McKean}, 
Taylor \cite{Taylor} and 
Merton \cite{Merton} there has been,
in recent times,  an important effort to characterize 
optimal stopping problems with one sided solutions, i.e. such that the 
optimal stopping time is a threshold stopping time, usually of the form
$$
\tau^*=\inf\{t\geq 0\colon X(t)\geq x^*\},
$$
for some critical $x^*$. In this situation, we say that the problem has a
\emph{one-sided} solution. 
Villeneuve \cite{Villeneuve}
gives sufficient conditions to have threshold optimal strategies, 
Arkin \cite{Arkin} gives necessary and sufficient conditions for It\^o diffusions with $C^2$ payoffs functions
to have one sided solutions, whereas
Arkin and Slastnikov \cite{ArkinSlastnikov} 
and Crocce and Mordecki \cite{CrocceMordecki}
give also necessary and sufficient conditions in different and more general diffusion frameworks.
%
%\subsection{general verification results}
 %
For more general Markov processes  Mordecki and Salminen \cite{MordeckiSalminen},
Christensen et al. \cite{ChristensenSalminenBao}, and Christensen and Irle \cite{ChIr2} 
propose verification results for one sided solutions, but also for problems where the optimal stopping time is of the form 
$$
\tau^*=\inf\{t\geq 0\colon X(t)\notin (x_*,x^*)\}.
$$
In this second situation is it is said that the problem has a \emph{two sided} solution.
For general reference of the theory of optimal stopping, see the books by Shiryaev \cite{Shiryaev} 
and by Peskir and Shiryaev \cite{PeskirShiryaev}.

%\subsection{Recent references}

%%%%%%%%%%%%%%%%%%%%%%%%
%%%%%% our approach
%%%%%%%%%%%%%%%%%%%%%%%%
\smallskip
\par\noindent
{\bf 1.3  Present study.}
{\red    In this paper we are interested to understand  situations where the stopping region
is disconnected due to the behavior of the underlying stochastic process and not due to the properties of
the payoff function. 
Such a case has been  found in \cite{AlvarezSalminen} with skew Brownian motion as the underlying. The idea is to study the shape of the continuation set as a function of the discount parameter $r.$ 
To depart, we consider a model where the solution of the optimal stopping problem is trivial in cases $r=0$ (no stopping) 
and $r=\infty$ (immediate stopping), and describe the stopping set as a function of $r$ as $r$ increases from 0 to $\infty.$ }
%
%The interesting situation that we try to understand is when the stopping region
%is not connected, but due to the nature of the stochastic process and not to
%the payoff function. This was the situation found in \cite{AlvarezSalminen}.  
%
When isolated continuation intervals appear when $r$ increases we say that  a \emph{shape transition} occurs. % with the variation of $r$, 
Such intervals are called \emph{bubbles}, for the definition, see  Section 2 % , (defined below) and say that we have a \emph{shape transition}.

The rest of the paper is organized as follows. In Section \ref{section:general} we present some 
preliminary general  results mainly on optimal stopping for diffusions. In Section \ref{section:broken}  the Brownian motion with broken drift as the solution of a stochastic differential
equation is introduced and its main characteristics are analyzed. In Section \ref{stopping} we solve the optimal stopping problem of the Brownian motion with broken drift with reward $x\mapsto (1+x)^+.$ In particular, it is seen that if the discontinuity in the drift is big enough a shape transition in the continuation region occurs. 
 
\section{Preliminary results}\label{section:general}
%Assume that the function $g$ satisfies the inversion formula, in the sense that
%$$
%g(x)=\int_{{\bf I}}G_r(x,y)(r-\LL)g(y)m(dy).
%$$
{\red 
From the general theory of Markovian stopping problems, the optimal stopping time $\tau^*$ in (\ref{eq:osp}), if such a time exists, can be characterized (see Theorem 3, Section 3.3 in \cite{Shiryaev}) as the first entrance time into the stopping set
\begin{equation}\label{eq:set}
{\bf S}(r)=\{x\in{\bf I}\colon V_r(x)=g(x)\},
\end{equation}
where we have indicated the dependence on the discounting parameter $r\geq 0.$ The complement of ${\bf S}(r)$, i.e., 
 ${\bf C}(r)={\bf I}\setminus {\bf S}(r)$ is called the continuation set.
  
\begin{proposition}\label{prop:10} 
Let $0\leq r_1<r_2$ be two discounting parameters and consider the corresponding OSPs as given in  (\ref{eq:osp}). Then 
\begin{equation}\label{eq:Vr}
V_{r_1}(x)\geq V_{r_2}(x)\quad {\rm for\ all}\  x\in{\bf I},
\end{equation}
and 
\begin{equation}\label{eq:Sr}
{\bf S}(r_1)\subseteq {\bf S}(r_2).
\end{equation}
\end{proposition}
\begin{proof} Let $\tau$ be a fixed stopping time. Then  % Because  $r_1<r_2$, for we have 
$$
e^{-r_1\tau}g(X_{\tau})\geq e^{-r_2\tau}g(X_{\tau}),
$$ 
and taking the suprema yields (\ref{eq:Vr}). Using the characterization of the stopping set in (\ref{eq:set}) we obtain now that  if $x\in {\bf S}(r_1)$ then $x\in {\bf S}(r_2)$ resulting to  (\ref{eq:Sr}).
\end{proof}
}

%\begin{proposition}\label{prop:1}
%Consider a diffusion $X$ with state space ${\bf I}.$
%\vskip1mm\par\noindent {\rm (b)}
%if finite, the number of connected components of the stopping set is increasing with $r$.
%\vskip1mm\par\noindent {\rm (c)} Under XXX (to be determined)
%There exists $r_0>0$ such that the stopping set is connected.
%\end{proposition}

%Consider a diffusion $X$ with state space ${\bf I}$,
%endpoints $a,b$, and infinitesimal generator $\LL$.
%Given $r\geq 0$ consider the measure 
%$$
%\sigma(dy)=(r-\LL)g(y)m(dy),
%$$ 
%has a Hahn decomposition $[a,b]={\bf N}(r)\cup {\bf P}(r)$. 
%We call ${\bf N}(r)$ the \emph{negative set} of the problem.
%it is known that 
%$$
%{\bf N}(r)\subset {\bf C}(r).
%$$
%\begin{proposition}  
%If ${\bf N}(r)=\{x\in{\bf I}\colon x\leq x_0\}$ then  the optimal stopping problem is one-sided.
%\end{proposition}
%\begin{proof}
%Consider
%$$
%m=\sup_{x\geq x_0}\int_{[{a},x)}\varphi_r(y)(r-\LL)g(dy)m(dy).
%$$ 
%If $m\leq 0$ then the process $e^{-rt}g(X(t))$ is a sub-martingale {\blue should we include a proof}f and the optimal stopping problem has a trivial solution
%$
%\tau^*=\infty.
%$
%In case  $m>0$ define
%$$
%x^*=\inf\{x\geq a\colon \int_{[{a},x)}\varphi_r(y)(r-\LL)g(dy)m(dy)\geq 0\}.
%$$
%Hence, we may argue that the stopping region for the problem is ${\bf S}(r)=[x^*,b]$.
%\end{proof}

{\red
%\noindent  
We recall next the smooth fit theorem from \cite{Salminen},  \cite{Peskir}, \cite{SalminenTa}. To fix ideas and  to focus on the case studied below, the theorem is formulated here for a left boundary point of the stopping region ${\bf S}(r).$

\begin{theorem}
\label{smoothfit1}
 Let $z$ be a left boundary point of ${\bf S}(r),$ i.e., 
$[z,z+\varepsilon_1)\subset {\bf S}(r)$ and $(z-\varepsilon_2,z)\subset
{\bf C}(r)$ for some positive $\varepsilon_1$ and  $\varepsilon_2.$
%Let $F$ be a continuous and increasing function and 
Assume that the
reward function 
$g$ and the fundamental solutions $\varphi_r$
and $\psi_r$ are  %$F$-
differentiable at $z.$ Then the value
function $V_r$ in (\ref{eq:osp})  is %$F$-
differentiable at $z$ and it holds $V^\prime(z)=g^\prime(z).$%the smooth fit with respect

\end{theorem}

Since the value function $V$ is bigger than the reward $g$ on the continuation set ${\bf C}(r)$ we introduce the following terminology. 

\begin{definition}\label{bubble}
A bounded open intervall $(x_1,x_2)\subseteq {\bf C}(r)$ is called a bubble if {\blue $ x_1, x_2\in {\bf S}(r).$} 

%there exists $\varepsilon >0$ such that $(x_1-\varepsilon,x_1] \cup [ x_2, x_2+\varepsilon) \subseteq {\bf S}(r). $
\end{definition}

\noindent
In \cite{AlvarezSalminen} it is seen that for skew Brownian motion and a large class of reward functions one can find a lower bound $r_0$  for the discounted parameter $r$ such that for all $r\geq r_0$ there is a bubble (containing the skew point). 
This is, in particular, true for the reward function $g(x)=(1+x)^+,$ i.e., no matter how big is the discounting it is not optimal to stop at the skew point.  In the present paper we study the appearance and the disappearance of a bubble for a Brownian motion with positive  drift changing at the origin. We define this process -  which we call a Brownian motion with broken drift - in the next section.   

To make the presentation more self-contained,  we display a result from \cite{AlvarezSalminen}, see Lemma 2 therein, which is  used to verify that a candidate solution of OSP (\ref{eq:osp}) is indeed the value function.  This is essentially Corollary on p. 124 in \cite{Shiryaev}. 
%We present the proof for readability and completeness.
\begin{proposition}\label{apu}
Let $A\subset \mathcal{I}$ be a nonempty Borel subset of $\mathcal{I}$ and 
$$
H_A=\inf\{t\geq 0:X_t\in A\}.
$$ 
Assume that the function
$$
\hat{V}(x):=\mathbb{E}_x\left[\textrm{e}^{-r\,H_A}g(X_{H_A})\right]
$$
is $r$-excessive and dominates $g$. Then $\hat{V}$ coincides with the value function of OSP (\ref{eq:osp}) and $H_A$ is an optimal stopping time.
\end{proposition}
%\begin{proof}
%Notice that since $X$ is recurrent and $A$ is nonempty it follows that $\tau_A<\infty$ almost surely.
%By the definition of $V$ it holds
%$$
%V(x)=\sup_{\tau\in \mathcal{T}} \mathbb{E}_x\left[\textrm{e}^{-r\tau}g(X_\tau)\right]\geq \mathbb{E}_x\left[\mathrm{e}^{-r\tau_A}g(X_{\tau_A})\right] =\hat{V}(x)
%$$
%for all $x$. On the other hand,  $\hat{V}$ being an $r$-excessive majorant of $g$ yields
%$$
%V(x)=\sup_{\tau\in \mathcal{T}} \mathbb{E}_x\left[\mathrm{e}^{-r\tau}g(X_\tau)\right]\leq \sup_{\tau\in \mathcal{T}} \mathbb{E}_x\left[\mathrm{e}^{-r\tau}\hat{V}(X_\tau)\right]\leq \hat{V}(x).
%$$
%Consequently, $V=\hat{V}$ and $\tau_A$ is an optimal stopping time.
%\end{proof}

The following technical result is needed in the proof of Proposition \ref{prop3}.
 
\begin{proposition}\label{apu2} Let   $h_{r_1},r_1>0,$ be an $r_1$-excessive function and 
$\{h_r\,;\, r_1\leq r<r_2\}, r_1<r_2,$  a sequence of functions such that $h_r$ is $r$-excessive ,
Assume that $h_r\leq h_{r_1}$ for $r\geq r_1$ and  $\lim_{r\uparrow r_2} h_r(x)=:h_{r_2}(x) $ exists  for all $x.$ Then $h_{r_2}$ is $r_2$-excessive.
\end{proposition}
\begin{proof}
Consider
\begin{align*}
\mathbb{E}_x\left[\textrm{e}^{-r_2t}h_{r_2}(X_t)\right]&=\mathbb{E}_x\big[\lim_{r\uparrow r_2}{\textrm{e}}^{-rt}h_{r}(X_t)\big]\\
&=\lim_{r\uparrow r_2}\mathbb{E}_x\left[{\textrm{e}}^{-rt}h_{r}(X_t)\right]\\ &\leq\lim_{r\uparrow r_2}h_r(x) \\
&=h_{r_2}(x),
\end{align*}
where in the second step we use the dominated convergence theorem which is applicable since ${\textrm{e}}^{-rt}h_r\leq {\textrm{e}}^{-r_1t}h_{r_1}$ and 
$$
%\mathbb{E}_x\left[{\textrm{e}}^{-rt}h_{r_1}(X_t)\right]\leq 
\mathbb{E}_x\left[{\textrm{e}}^{-r_1t}h_{r_1}(X_t)\right]\leq h_{r_1}(x)<\infty.
$$
\end{proof}

}

%to $F$   holds:
%\begin{equation}
%\label{sfeq}
%\frac{d^+V}{dF}(z)= \frac{d^-V}{dF}(z)=\frac{dg}{dF}(z).
%\end{equation}

\section{Brownian motion with broken drift}\label{section:broken}
%\subsection{Definition and basic facts}
Consider a diffusion 
\begin{equation*}
    X(t)=x+\int_0^t \mu(X(s))ds+W(t),
\end{equation*}
where
$$
\mu(x)=
\begin{cases}
{\mu_1},&\text{ for $x<0$},\\
{\mu_2},&\text{ for $x\geq 0$},
\end{cases}
$$
and {\red $0\leq \mu_1< \mu_2.$} 
The speed measure of this diffusion is given by
$$
m(dx)=
\begin{cases}
2e^{2\mu_1x}dx,&\text{ for $x<0$},\\
2e^{2\mu_2x}dx,&\text{ for $x> 0$},
\end{cases}
$$
while the scale function is
$$
S(x)=
\begin{cases}
{1\over 2\mu_1}(1-e^{-2\mu_1x}),&\text{ for $x<0$},\\
{1\over 2\mu_2}(1-e^{-2\mu_2x}),&\text{ for $x\geq 0$}.
\end{cases}
$$
We call the diffusion $\{X_t\colon t\geq 0\}$ a Brownian motion with broken drift, and remark that in the literature one can also find a diffusion called the Brownian motion with alternating drift (also the bang-bang Brownian motion), see  
\cite{BorodinSalminen} p. 128 and references therein. 
Notice that the scale function is differentiable everywhere with the derivative %. For further reference we compute its derivative:}
$$
S'(x)=
\begin{cases}
e^{-2\mu_1x},&\text{ for $x<0$},\\
e^{-2\mu_2x},&\text{ for $x\geq 0$}.
\end{cases}
$$
We find next  the fundamental solutions for $r>0.$  Introduce
\begin{align*}
\lambda^-_1&=-\sqrt{\mu_1^2+2r}-{\mu_1}<0,
&\lambda^+_1=\sqrt{\mu_1^2+2r}-{\mu_1}>0,\\
\lambda^-_2&=-\sqrt{{\mu_2}^2+2r}-{\mu_2}<0,
&\lambda^+_2=\sqrt{{\mu_2}^2+2r}-{\mu_2}>0.\\
\end{align*}
The decreasing fundamental solution is
\begin{equation}\label{eq:phi}
\varphi_r(x)=
\begin{cases}
A_1\exp(\lambda_1^-x)+A_2\exp(\lambda_1^+x),&\text{ for $x<0$},\\
\exp(\lambda_2^-x),&\text{ for $x\geq 0$,}\\
\end{cases}
\end{equation}
where the constants $A_1$ and $A_2$ are determined so that $\varphi_r$ is  differentiable at 0. Hence,
\begin{align*}
A_1&={\lambda_1^+-\lambda_2^-\over \lambda_1^+-\lambda_1^-}={\lambda_1^+-\lambda_2^-\over
2\sqrt{\mu_1^2+2r}}>0,
\qquad
A_2={\lambda_2^- -\lambda_1^-\over \lambda_1^+-\lambda_1^-}={\lambda_2^- -\lambda_1^-\over 
2\sqrt{\mu_1^2+2r}}{\red <0.}
\end{align*}
Observe that $A_1+A_2=1$. %Furthermore, the function $\varphi_r$ is differentiable {\red everywhere, in particular } at $x=0$.
Analogously, the increasing solution is 
\begin{equation}\label{eq:psi}
\psi_r(x)=
\begin{cases}
\exp(\lambda_1^+x),&\text{ for $x<0$,}\\
B_1\exp(\lambda_2^+x)+B_2\exp(\lambda_2^-x),&\text{ for $x\geq 0$}
\end{cases}
\end{equation}
with
$$
B_1={\lambda_1^+-\lambda_2^-\over \lambda_2^+-\lambda_2^-}=
{\lambda_1^+-\lambda_2^-\over 2\sqrt{\mu_2^2+2r}}>0,
\qquad
B_2={\red{\lambda_2^+ -\lambda_1^+\over \lambda_2^+-\lambda_2^-}
={\lambda_2^+ -\lambda_1^+\over 2\sqrt{\mu_2^2+2r}} <0.}
$$
{\red The above stated properties $A_1>0, B_1>0,$ and $A_2<0$ are easily seen from the explicit expressions for $\lambda^+_i, \lambda^-_i, i=1,2$. For $B_2<0$ notice that $\mu\mapsto\sqrt{\mu^2+2r}-\mu$ is decreasing when $\mu>0.$ In Figure \ref{figure:1} we have visualized $\varphi_r$ and $\psi_r$ when $r=3$, ${\mu_1}=1,{\mu_2}=10.$} 
\begin{figure}[h]
\centering
\includegraphics[scale=0.18]{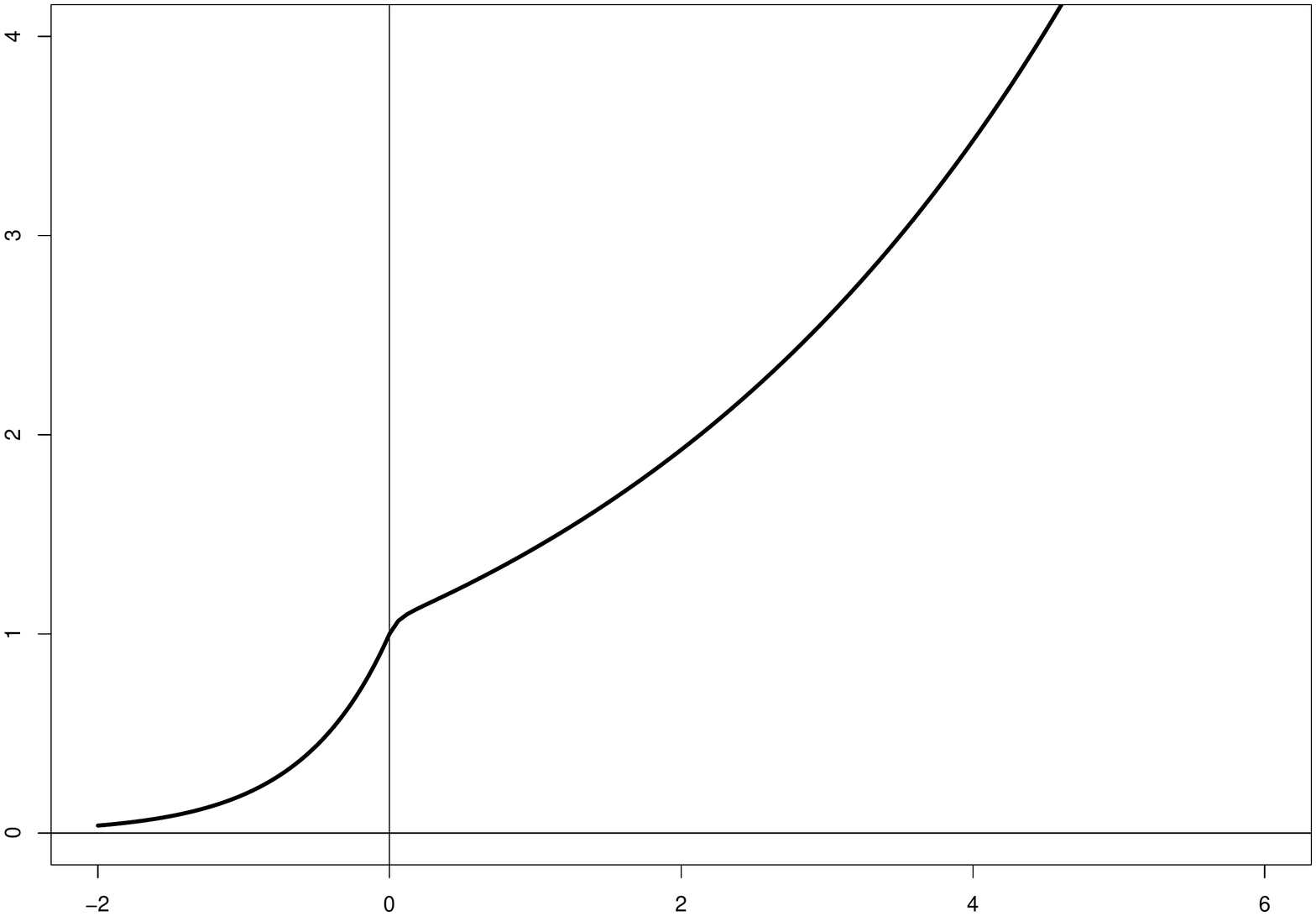}
\includegraphics[scale=0.18]{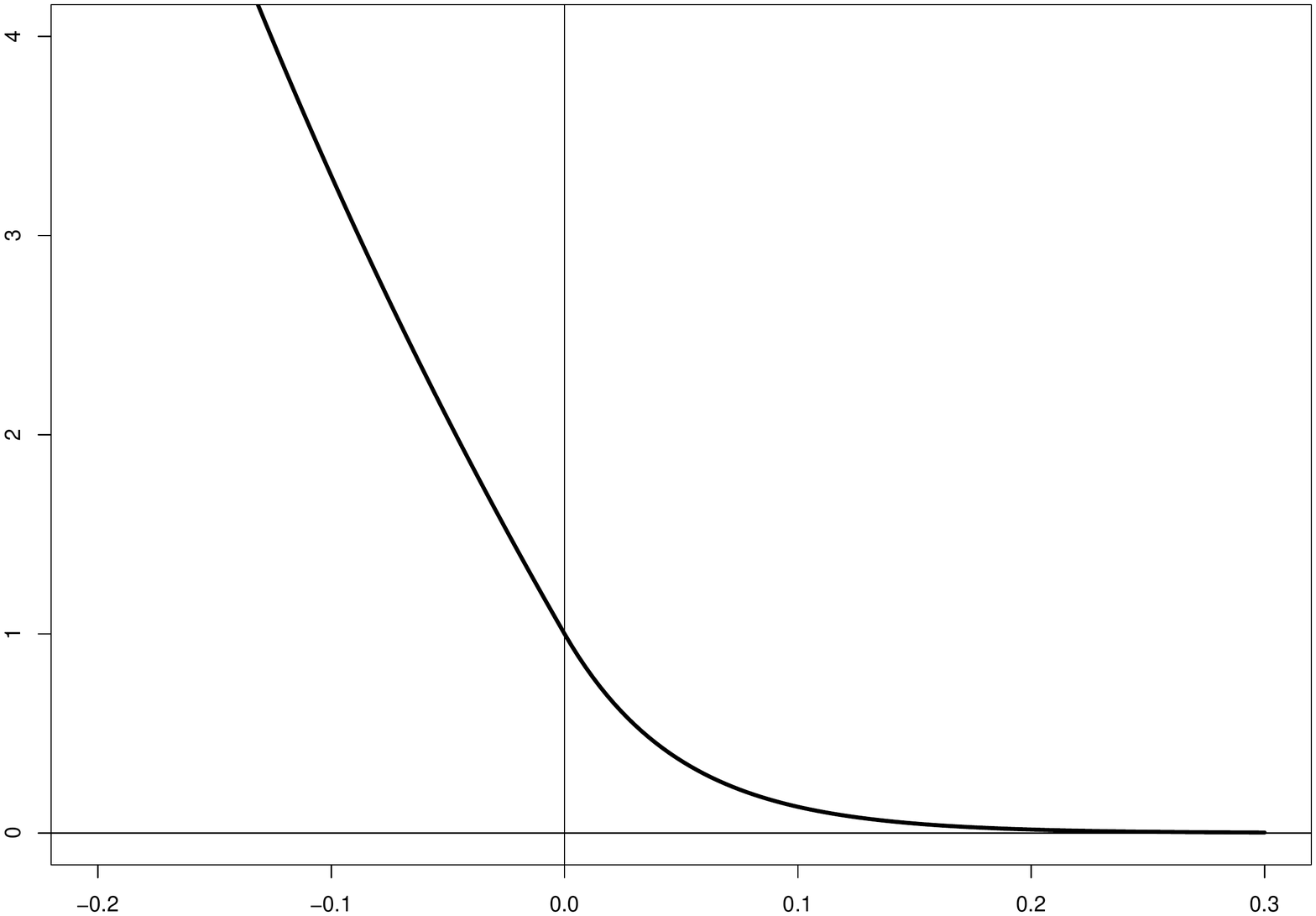}
\caption{The functions $\psi_r$ (left) and $\varphi_r$ (right), 
for $r=3$, ${\mu_1}=1,{\mu_2}=10$.}
\label{figure:1}
\end{figure}

\section{Solution of the optimal stopping problem}\label{stopping}
We analyze the optimal stopping problem \eqref{eq:osp} for the {\red  broken-drift} diffusion introduced {\red above} and the reward function $g(x)=(1+x)^+$, i.e.,
\begin{equation}\label{osp1}
 \sup_{\tau \in \mathcal{M}}\E_{x}\big[e^{-r\tau} g(X_\tau)\big].
\end{equation}
{\red The main issue is to show that there are values on $\mu_1$ and $\mu_2$ such that, for some values on $r,$ the continuation region is disconnected and contains $0,$ i.e., the point  where the drift changes. However, letting  
here $r$ to increase makes the continuation region connected and then  0 is  in the stopping set. 
This situation is different from the one studied in \cite{AlvarezSalminen} where the skew point 0 is for all values of $r$ and the skewness parameter $\beta> 1/2$ in the continuation set. %We consider first the case
%\begin{equation}\label{mus}
%0\leq \mu_1<\mu_2
 %\end{equation}
}
%The discussion is whether we can have a disconnected stopping region. 

Clearly, it follows from the assumption $0\leq \mu_1<\mu_2$ that  $X(t)\to +\infty$ as $t\to +\infty.$ Consequently, to make 
OSP \eqref{eq:osp} non-trivial we assume that $r>0.$ Because $g(x)=0$ for $x\leq -1$ it holds  
$(-\infty, -1)\subset {\bf C}(r).$ Notice also that the smooth fit theorem applies for all values of $r,$ i.e., the value function meets the reward smoothly at every boundary point between ${\bf C}(r)$ and ${\bf S}(r).$

For the analysis to follow, we define (cf. \cite{Salminen})
\begin{align}\label{G-}
G_-(x)&:=\left(\psi_r'(x)(x+1)-\psi_r(x)\right)/S'(x)\\
\label{G+}
G_+(x)&:=\left(\varphi_r(x)-(1+x)\varphi_r'(x)\right)/S'(x),
\end{align} 
and their  derivatives with respect to the speed measure for $x>-1$ and $x\not= 0$ are  given by
\begin{align}\label{dG-}
\nonumber G'_-(x)&=m(x)\frac d{dm}G_-(x)=m(x)\psi_r(x)\left(r(1+x)-\frac d{dm}\frac d{dS} (1+x)\right)\\
&=m(x)\psi_r(x)\begin{cases}r(1+x)-\mu_1,& x<0,\\r(1+x)-\mu_2,& x>0,\\\end{cases}
\end{align}
and, similarly, 
\begin{align}\label{dG+}
G'_+(x)%&=m(x)\frac d{dm}G_-(x)=m(x)\varphi_r(x)\left(r(1+x)-\frac d{dm}\frac d{dS} (1+x)\right)\\
&=m(x)\varphi_r(x)\begin{cases}\mu_1-r(1+x),& x<0,\\ \mu_2-r(1+x),& x>0,\\\end{cases}
\end{align}
where we have used that the fact that  $\varphi_r$ and $\psi_r$ solve the differential equation
$$
\frac d{dm}\frac d{dS} u=ru.
$$   
The functions $G_-$ and $G_+$ are used to check the excessivity of a proposed value function. An alternative way is to evoke the generalization of the Ito formula developed in Peskir \cite{Peskir2005}.

{\red 
\begin{proposition} \label{prop1} In case $0<r\leq\mu_1\leq\mu_2$ the continuation region for  OSP
(\ref{osp1}) is given by
$$
{\bf C}(r)=(-\infty,c),
$$
where $c=c(r)>0$ is the unique solution of the equation
\begin{equation}\label{eqG}
\psi_r'(x)(x+1)-\psi_r(x)=0.
\end{equation}
\end{proposition} 

\begin{proof} We show first that equation (\ref{eqG}) has a unique positive solution. For this consider for $x>-1$  the function 
$G_-$ defined in \eqref{G-}.
%\begin{align}\label{G-}
%G_-(x)&:=\left(\psi_r'(x)(x+1)-\psi_r(x)\right)/S'(x).
%\end{align}
 Since $S'(x)>0$ for all $x$ the claim is equivalent with the statement that  $G_-$ has a unique positive zero. In fact, we prove a bit more; namely that  $G_-$ attains the global minimum at $x_0:=(\mu_2-r)/r>0,$ is negative and decreasing for $x\leq x_0,$ is increasing for $x>x_0,$ and has, therefore,  a unique zero. Analyzing $G'_-$ as given in \eqref{dG-}, 
 it is straightforward to deduce, since $0<r\leq\mu_1\leq\mu_2,$ the claimed properties of $G_-.$ Let 
$$
H_c:=\inf\{t\, :\, X_t\geq c\},
$$
where $c$ is the unique solution of (\ref{eqG}), and define
\begin{equation}\label{hatV}
\widehat V(x):=\E_{x}\big[e^{-r H_c} g(X_{H_c})\big]=\begin{cases}\displaystyle{\frac{\psi_r(x)}{\psi_r(c)}}\,g(c),& x\leq c,\\ 
g(x) ,& x\geq c.\\
\end{cases}
\end{equation}
If $\widehat V$ is an $r$-excessive majorant of $g$ it follows from Proposition \ref{apu2} that $\widehat V$ is the value function of OSP (\ref{osp1}).  The excessivity can be checked with the method based on the representation theory of excessive functions (cf. \cite{Salminen} Section 3).   This boils down to study for $x\not= -1$ the functions 
\begin{align}\label{D-}
I_V(x)&:=\left(\psi'_r(x)\widehat V(x)-\psi_r(x)\widehat V'(x)\right)/S'(x),\\
\label{I+}
D_V(x)&:=\left(\varphi_r(x)\widehat V'(x)-\widehat V(x)\varphi_r'(x)\right)/S'(x).
\end{align}
Clearly, $I_V(x)=0$ for $x\leq c$ and increasing for $x>c.$ Notice that $I_V=G_-$ on $[c,+\infty).$  Studying the derivative (with respect to the speed measure)  of $D_V$ it is easily seen that $D_V$ is positive and decreasing to 0 on $[c,+\infty).$ 
Consequently, $I_V$ and $D_V$ induce a (probability) measure which represent $\widehat V$ proving that   $\widehat V$ is $r$-excessive. To prove that $\widehat V$ is a majorant of $g$ consider for $-1<x<c$  
$$
\widehat V(x)\geq g(x) \quad \Leftrightarrow\quad \frac{\psi_r(x)}{g(x)}\geq\frac{\psi_r(c)}{g(c)}.
$$
The right hand side of this equivalence holds since the derivative of $x\mapsto \psi_r(x)/g(x)$ is $G_-$ which is negative for  $-1<x<c,$ as is shown above.
\end{proof}

In case  $\mu_1=\mu_2$ it is well-known  (see \cite{Taylor}, and \cite{Salminen} where the problem is solved using the representation theory of the excessive functions) that ${\bf S}(r)=[c,+\infty)$ with $c=c(r)$ as in Proposition \ref{prop1}, i.e., 
$$
c=\frac {1}{\lambda_1^+}-1.
$$  
Consequently, it is expected that if $\mu_2$ is relatively close to $\mu_1$ the stopping region is of this form for all values of the discounting parameter $r;$ in other words, there 	is no bubble. This is indeed the case and the exact formulation is as follows.

\begin{proposition} \label{prop2} In case  $0\leq\mu_1\leq\mu_2\leq \mu_1+{\frac 12}$ the continuation region for the OSP (\ref{osp1})
%\begin{equation}\label{eq:osp1}
 %\sup_{\tau \in \mathcal{M}}\E_{x}{e^{-r\tau} (1+X_{\tau})^+}.
%\end{equation}
is given by
$$
{\bf C}(r)=(-\infty,c),
$$
where $c=c(r)$ is the unique solution of equation (\ref{eqG}):
$$
\psi_r'(x)(x+1)-\psi_r(x)=0.
$$
As $r$ increases from 0 to $+\infty$ then $c(r)$ decreases monotonically  from $+\infty$ to $-1.$ In particular, $c(r)=0$ for $r=\mu_1+{\frac 12}.$ 
\end{proposition} 
\begin{proof} If $r\leq\mu_1$ (and $\mu_1>0$) the statement is the same as in Proposition \ref{prop1}. We assume next that $r\geq \mu_2.$  The proof in this case is very similar to the proof of  Proposition \ref{prop1}. It can be proved that  $G_-$ attains the global minimum at $x_1:=(\mu_1-r)/r<0,$ is negative and decreasing for $x\leq x_1,$ is increasing for $x>x_1,$ and has, therefore,  a unique zero. Consequently, this root can be taken to be an optimal stopping point $c=c(r)$ and the analogous function $\widehat V$ as in (\ref{hatV}) can be proved to be the value of OSP (\ref{osp1}). 
Finally, assume $\mu_1<r<\mu_2.$  In this case, $G_-$ has a local maximum at 0, which is negative since 
$$
G_-(0)=\psi'_r(0)-\psi_r(0)=\lambda^+_1-1<0 \quad \Leftrightarrow\quad r<\mu_1+1/2.
$$
Hence, equation (\ref{eqG}) has a unique positive root and the proof can be completed as in the previous cases.  
\end{proof}

Finally, we study the situation  $0\leq \mu_1<\mu_1+\frac 12<\mu_2.$ The main observation is that there exists a bounded  interval such that if $r$  is in this interval then the continuation set has a bubble.  The first result concerns  the localization of a possible bubble, and is perhaps intuitively obvious. Anyway, we present its proof since the result  is needed when proving Proposition \ref{prop3} below which characterizes the  continuation (and the  stopping) set in the present case.   

\begin{proposition} \label{prop:bubble} Assume $r>0$ is such that $(c_1(r), c_2(r))$ is a bubble. Then it holds that $0\in [c_1(r), c_2(r)),$ and there exists at most one bubble.
\end{proposition} 

\begin{proof} Let  $G_-$ and $G_+$ be given as in (\ref{G-}) and  (\ref{G+}), respectively. Since $g,$ $\varphi_r$ and $\psi_r$ are differentiable everywhere we may apply  (4.7) Theorem in \cite{Salminen} p. 95 to deduce that $c_1=c_1(r)$ and $c_2=c_2(r)$ satisfy 
\begin{align*}
 G_-(c_1) &= G_-(c_2),\\
  G_+(c_1) &= G_+(c_2).
\end{align*}
Moreover, $G_-$ and $G_+$ are positive, non-decreasing and non-increasing , respectively, on  $(c_1-\varepsilon,c_1] \cup [ c_2, c_2+\varepsilon) \subseteq {\bf S}(r)$ for some $\varepsilon >0.$  Studying these explicit expressions of $G_-^\prime $ and $G_+^\prime$ given in (\ref{dG-}) and  (\ref{dG+}), respectively,  we conclude that necessarily $c_1\leq 0<c_2,$ as claimed. 
\end{proof}

Recall that the principle of smooth fit holds for our stopping problem.  Hence, it is enlightening to investigate which ``good'' candidates satisfying the smooth fit principle  cannot be value functions since they fail to be excessive. The following result shows that for $\mu_1+\frac 12\leq r<\mu_2$ there exist smooth fit (at 0)  functions which are harmonic on $\R_-$ but which are not $r$-execessive.   

\begin{proposition} \label{at0} For  $r\geq \mu_1+\frac 12$ there exist $A$ and $B$ such that the function
\begin{equation}\label{sfF}
F(x):=\begin{cases} A\exp(\lambda_1^+x)+B\exp(\lambda_1^-x),& x\leq 0,\\ 1+x,& x\geq 0,\end{cases}
\end{equation}
satisfies the principle of smooth fit at 0, i.e., $F'(0-)=F'(0+)=1.$ The function $F$ is $r$-harmonic (and positive) on  $(-\infty,0)$ but not $r$-excessive if  $r<\mu_2.$ For $r< \mu_1+\frac 12$ the coefficient $B$ is negative and the function $F(x)\to -\infty$ as $x\to -\infty$ (and the function is not  $r$-excessive).
\end{proposition}
%haa
\begin{proof} We study only the case $r=r_0:=\mu_1+\frac 12$ and leave the details of the other cases to the reader. In this case $\lambda^+_1 =\sqrt{\mu_1^2+2r}-{\mu_1}=1,$ and, obviously, 
$$ 
F(x):=\begin{cases} \exp(x),& x\leq 0,\\ 1+x,& x\geq 0,\end{cases}
$$
satisfies smooth fit at 0. %if $\lambda^+_1 =\sqrt{\mu_1^2+2r}-{\mu_1}=1.$ This means that $r=r_0=\mu_1+\frac 12.$ 
 Consequently, $F$ is $r_0$-harmonic (and positive) on $(-\infty,0)$ and  it remains to prove that $F$ is not $r_0$-excessive. For this,  consider the representing function (this corresponds $G_-$ in (\ref{G-}))
$$
x\mapsto \left(\psi_{r_0}'(x)F(x)-\psi_{r_0}(x)F'(x)\right)/S'(x).
$$
The claim is that this function is not non-decreasing. Indeed, take derivative with respect to the speed measure to obtain
\begin{align*}
&\frac{d}{dm}\Big(\left(\psi_{r_0}'(x)F(x)-\psi_{r_0}(x)F'(x)\right)/S'(x)\Big)\\
&\hskip3cm=F(x)\frac{d}{dm}\frac{d}{dS}\psi_{r_0}(x)-\psi_{r_0}(x)\frac{d}{dm}\frac{d}{dS}F(x)\\
&\hskip3cm=\psi_{r_0}(x)\begin{cases} 0,& x<0,\\
{r_0}(1+x)-\mu_2,& x>0.\\
\end{cases}
\end{align*}
Since ${r_0}<\mu_2$ this derivative is negative, e.g.,  for small positive $x$-values; therefore,  $F$ is not ${r_0}$-excessive.   
 \end{proof}

\begin{remark}
 From Proposition \ref{prop2} and \ref{prop3} we may conclude  that if $r\geq \mu_2$ ($r\leq \mu_1+1/2$) then the problem is  one-sided and the optimal stopping point is negative (positive). Notice that the smooth fit function $F$ in (\ref{sfF}) could be excessive for $r\geq\mu_2$ but since the optimal stopping point is negative $F$ is not the smallest excessive majorant of the reward.
\end{remark}

Next proposition can be seen as our main result concerning OSP (\ref{osp1}). It is proved that if  $r\in(\mu_1+1/2, \mu_2)$ but is ``close to'' $\mu_2$ then ${\bf C}(r)$ has a bubble.  However, the bubble disappears when  $r$ becomes bigger than $\mu_2$ or tends to $\mu_1+1/2.$  We give a complete description of ${\bf C}(r)$ although there is some overlap with Proposition \ref{prop2}. 

\begin{proposition} \label{prop3} In case  $0\leq\mu_1< \mu_1+{\frac 12}< \mu_2$ there exists $r_0\in(\mu_1+1/2, \mu_2)$ with the following properties:  
\begin{description}
 \item{(i)}\hskip.2 cm  If $r\in[r_0,\mu_2)$ the continuation region is given by
 $$
{\bf C}(r)=(-\infty,c_1)\cup (c_2,c_3),
$$
where $c_i=c_i(r),\, i=1,2,3,$ are such that $c_3> 0\geq c_2\geq c_1>-1.$ In particular, for $r=r_0$ it holds $c_1=c_2<0.$
\item{(ii)}\hskip.2 cm If $r\geq \mu_2$ the continuation region is given by 
 $$
{\bf C}(r)=(-\infty,c_-),%\quad{\rm where}\ -1< c<0. 
$$
where $c_-=c_-(r)<0$ is the unique solution of (\ref{eqG}). In particular, for $r=\mu_2$
\begin{equation}
\label{cc2}
c_-(\mu_2)=\frac{1}{\lambda^+_1(\mu_2)}-1 =(\sqrt{\mu_1^2+2\mu_2}+\mu_1-2\mu_2 )/2\mu_2<0.
\end{equation}
\item{(iii)}\hskip.2 cm If $r<r_0$ the continuation region is given by 
 $$
{\bf C}(r)=(-\infty,c_+),%\quad{\rm where}\ c>0. 
$$
where $c_+=c_+(r)>0$ is the unique solution of (\ref{eqG}).
 \end{description}

\end{proposition} 

\begin{proof} The proof of  (ii) is as the proof of Proposition \ref{prop2} when $r\geq \mu_2$. Notice, however, that in the present case $c(r)<0$ for all $r\geq \mu_2.$  We consider next (iii) in case $r\leq \mu_1+\frac 12.$ Studying the derivative of $G_-$ and the value of $G_-$ at zero  it is seen, as in the proof of Proposition \ref{prop1}, that equation $G_-(x)=0$ has for $r<\mu_1+\frac 12$ one (and only one) root $\rho=\rho(r)>0.$  In case $r=\mu_1+\frac 12$ there are two roots $\rho_1=\rho_1(r)=0$ and $\rho_2=\rho_2(r)>0.$ Proceeding as in the proof of Proposition \ref{prop1} it is seen that the stopping region is as claimed with $c_+=\rho$ if $r<\mu_1+\frac 12$ and $c_+=\rho_2$ if $r=\mu_1+\frac 12$.  Assume now that there does not exist a bubble for any  $r\in[\mu_1+\frac 12,\mu_2].$  Then for all $r\in[\mu_1+\frac 12,\mu_2]$ we can find  $c=c(r)$ such that ${\bf S}(r)=[c,+\infty).$ Knowing that $c(r) >0$ for $r=\mu_1+\frac 12$ and $c(r)<0$ for $r=\mu_2$ we remark first there does not exists $r$ such that $c(r)=0.$ Indeed, by Theorem \ref{smoothfit1} the value should satisfy the smooth fit principle at 0 but from Proposition \ref{at0}  we know  that such functions are not $r$-excessive. Next, using  ${\bf S}(r_1)\subseteq {\bf S}(r_2)$ for $r_1<r_2$  (cf. Proposition \ref{prop:10})  it is seen that $r\mapsto c(r)$ is non-increasing, and has, hence, left and right limits. Consquently, there exists a unique point $\hat r$  such that 
$$
\hat c_+:=\lim_{r\uparrow \hat r}c(r)>0 \quad {\rm and}\quad \hat c_-:=\lim_{r\downarrow \hat r}c(r)<0.
$$ 
Under the assumption that there is no bubble the value function is of the form given in (\ref{hatV}), i.e., 
\begin{align}\label{hatVr} 
V_r(x)&=\begin{cases}\psi_r(x)\frac{1+c(r)}{\psi_r(c(r))},& x\leq c(r),\\ 
1+x ,& x\geq c(r).\\
\end{cases}\\
\nonumber
&=\E_x\big(e^{-rH_c}\, (1+X_{H_c})\big),
\end{align}  
 where $H_c:=\inf\{y\,:\,X_t\geq c(r)\}.$    Letting in (\ref{hatVr})  $r\uparrow \hat r$ yields by Proposition \ref{apu2} an $\hat r$-excessive function which by Proposition  \ref{apu} is the value of the corresponding OSP (\ref{osp1}). Similarly, letting $r\downarrow \hat r$ yields an  $\hat r$-excessive function which should also be the value of the same OSP. However, the functions are clearly different and since the value is unique we have reached a contradiction showing that there exists at least one bubble.  Evoking  Proposition \ref{prop:bubble} completes the proof.      
\end{proof}

\begin{remark}
%{\bf 1.} 
The fact  that there is a bubble when $r<\mu_2$ but ``close'' to $\mu_2$  would also follow if we can prove that $G_-$ has a unique negative zero. Notice that  $G_-$ in this case is not monotone around 0. This would then imply the exisistence of a bubble if we can verify that  the local minimum on $(0,+\infty)$ is positive. However, we have not been able to show this. Numerical calculations with some parameter values give evidence that the local minimum on $(0,+\infty)$ is indeed positive.  
\end{remark}
\begin{remark}
%\noindent{\bf 2.} 
For $r<r_0$ we have the value function (cf. (\ref{hatV}))
\begin{equation}\label{hatV2} 
V_r(x)=\begin{cases}\psi_r(x)\frac{1+c_+}{\psi_r(c_+)},& x\leq c_+,\\ 
1+x ,& x\geq c_+.\\
\end{cases}
\end{equation}
Since $c_+>0$ it holds $V_r(0)>g(0),$ and, hence, $(1+c_+)/\psi_r(c_+)>1.$ Consequently, 
$$
V_r'(0)=\lambda^+_1\frac{1+c_+}{\psi_r(c_+)}>1,
$$
because also $\lambda^+_1>1$ for $r>\mu_1+\frac 12.$ Moreover, $V_r'(x)\to 0$ as $x\to -\infty,$ and there exists a unique point $a=a(r)$ such that $V_r'(a)=1.$ From Proposition \ref{prop3}  we know that the bubble appears as $r$ increases and takes the value $r_0.$  Therefore, we may describe the value function $V_{r_0}$ to be of the form in (\ref{hatV2}) satisfying the smooth fit at $c_+(r_0)>0$ and also at another point $a(r_0)<0$ which is a tangent point with the reward.      
\end{remark}

\noindent
{\bf Acknowledgement.} This research has been partially supported by grants from Magnus Ehrnrooths stiftelse, Finland.   

%hii

}

\end{document}